\documentclass[11pt]{amsart}
\usepackage{amsmath}
\usepackage{amssymb}
\usepackage{mathrsfs}
\usepackage{latexsym}
\usepackage{amscd}
\usepackage{epsfig}
\usepackage{graphicx}
\usepackage{epstopdf}

\newtheorem{theorem}{Theorem}[section]
\newtheorem{mtheorem}{Main Theorem}
\newtheorem{defin}{Definition}[section]
\newtheorem{lemma}{Lemma}[section]

\newtheorem{remark}{Remark}

\newtheorem{ex*}{Example}

\newcommand{\real}{\mathbb{R}}
\newcommand{\comp}{\mathbb{C}}
\newcommand{\realp}{\mathbb{R}_{+}}
\newcommand{\inte}{\mathbb{Z}}

\newcommand{\dd}[1]{\frac{\partial}{\partial #1}}

\newcommand{\half}{\frac{1}{2}}

\newcommand{\torus}{\mathbb{T}}
\newcommand{\ctorus}{(\comp^{*})}

\begin{document}
\title{Lattice points counting via Einstein metrics}
\author{Naichung Conan Leung and Ziming Nikolas Ma}
\address{The Institute of Mathematical Sciences and Department of Mathematics, The Chinese University of Hong Kong, Shatin, Hong Kong}
\email{leung@math.cuhk.edu.hk}
\email{ncma@math.cuhk.edu.hk}
\date{}
\maketitle

\begin{center}\textbf{Abstract}
\end{center}

We obtain a growth estimate for the number of lattice points inside any $\mathbb{Q}$-Gorenstein cone. Our proof uses the result of Futaki-Ono-Wang on Sasaki-Einstein metric for the toric Sasakian manifold associated to the cone, a Yau's inequality, and the Kawasaki-Riemann-Roch formula for orbifolds.
\section{Introduction}
The Ehrhart polynomial $p_{P} : \mathbb{Z} \rightarrow \mathbb{Z}$ associated to a lattice polytope $P$ inside an $n$-dimensional latticed vector space $\mathbb{Z}^n \subset \real^n$ is given by 
$$p_P(k)= \#(k P \cap \mathbb{Z}^n)=\sum_{i=0}^n a_i k^i.$$ Lots of work has been done to get estimates of the polynomial, using either combinatorial or geometric methods (see for example \cite{CtCD}). In this paper, we are interested in obtaining a lower estimate of $p_P(k)$ for large $k$ using toric geometry and Einstein metrics.\\

When the polytope is Delz\'ant, the Ehrhart polynomial has an expression via toric geometry, by associating to $P$ a toric manifold $X_P$ and using the Riemann-Roch formula. It is well known that the leading coefficient $a_n$ is $Vol(P)$ and $a_{n-1}$ is determined by $Vol(P)$ if $P$ is reflexive. A lower estimate is obtained by considering $a_{n-2}$, which is an integral of the second Chern class of $X_P$, besides terms involving volume. When the polytope is balanced (i.e. the center of mass agrees with the origin), we can obtain an estimate $$a_{n-2} \geq \frac{(3n+2)(n-1)n}{24(n+1)}  Vol(P),$$ using the existence of the K\"ahler-Einstein metric (see \cite{Yau78} and \cite{WZ04}) and a Yau's inequality (see appendix \ref{pointchernnoieq}).\\

In this paper, we want to generalize this result to reflexive polytopes which may not be balanced. Given any polytope $P$, we can form its cone $C^\lor(P)=cone(P\times \{1\})\subset \real^{n}\times \real$ and let $\xi = (\vec{0},1)$. We count the number of lattice points up to level $k$ defined by $\xi$, that is, $n_\xi(k)= \# \{ x \in C^\lor(P) \cap \mathbb{Z}^{n+1} | (x,\xi) \leq k \}$. The two counting functions are related by $$n_\xi(k)= \sum_{i=0}^{k} p_P(i).$$ We will reformulate the counting problem by considering $n_\xi(k)$.\\

From now on, instead of using the standard lattice $\mathbb{Z}^n \times \mathbb{Z}$, we let $N$ be any rank $n+1$ lattice and $M$ be its dual. Let $C^{\lor} \subset M_{\real}$ be a cone. We can choose an affine hyperplane $H_\xi=\{x\in M_{\real} | ( x,\xi) = 1\}$ by picking a dual vector $\xi \in N_\real$. The hyperplane is moved toward infinity by changing $\xi$ to  $\xi/k$ and letting $k \rightarrow \infty$. We consider the function 
$$\begin{array}{rl}
n_\xi(k) &= \# \{x \in C^\lor \cap M | (x,\xi) \leq k\}\\
\\
& = b_{n+1} k^{n+1} + b_n k^n + b_{n-1}k^{n-1} + \mathcal{O}(k^{n-2})
\end{array}$$ 
which counts the number of lattice points inside the cone $C^\lor$ below the affine hyperplane $H_{\xi/k}$ (see \S 2, Figure \ref{fig:conefigure}). Similar to the polytope case, the first two coefficients, $b_{n+1}$ and $b_n$, are related to the volume of a certain polytope $\Delta_\xi$ determined by the $\xi$. Therefore, we focus on the first non-trivial coefficient $b_{n-1}$.\\

The advantage of this formulation is that we have the freedom to rotate the hyperplane by changing $\xi$. If the cone is $\mathbb{Q}$-Gorenstein (see Definition \ref{defCYcone}), we always have a balancing direction $\xi_c$ (see appendix \ref{thmxic}), which satisfies $$n_\xi(k) \geq n_{\xi_c}(k)\;\;for\;\;k\gg0,$$ for any other normalized vector $\xi$ in the interior of the cone $C$. $\xi_c$ will play the role of a balancing direction for the cone $C^\lor$. For $\xi$ close enough to $\xi_c$, the coefficient $b_{n-1}$ of $n_{\xi}(k)$ will have a lower estimate.\\

In the case that $\xi_c$ is a rational vector and the polytope $\{x \in C^\lor |(x,\xi_c)=1\}$ is Delz\'ant, our result gives
$$\begin{array}{rl}
 \frac{n_\xi(k)}{Vol_{n+1}(\Delta_\xi)} \geq&   k^{n+1}+\frac{(n+1)(n+2)}{2} k^n + \frac{n(n+1)(n+2)(3n+5)}{24}k^{n-1} + \mathcal{O}(k^{n-2}).
 \end{array}$$
In general, we have the following main theorem:
\begin{mtheorem}
Given an $(n+1)$-dimensional $\mathbb{Q}$-Gorenstein cone $C^\lor \subset M_\real^{n+1}$,
 with its canonical Reeb vector $\xi_c \in C \subset N_\real$,
 if $\xi_c$ is rational, let $\xi \in N$ be a primitive vector parallel to it; otherwise, choose $\xi$ having its direction close enough to $\xi_c$. Then\\
$$\begin{array}{rl}
 b_{n+1} = &\displaystyle Vol_{n+1}(\Delta_\xi)\\
 \\
 b_n =&\displaystyle \frac{1+q}{2}(n+1) Vol_{n+1}(\Delta_\xi)\\
 \\
 b_{n-1} \geq &\displaystyle c_{q,n} Vol_{n+1}(\Delta_\xi)\\
 \\
 &\displaystyle + \sum_{\rho \in C(1)} c_{\rho,n}Vol_{n}(H_\rho)

\end{array}$$

Here $\Delta_\xi=\{(\xi,y) \leq 1\} \subset M_\real$ is the polytope cut out by $\xi$ and $H_\rho = \{(\xi,y) \leq 1\}\;\cap\; \rho^{\perp}\cap C^\lor \subset \Delta_\xi$ is the corresponding facet associated to each ray $\rho \in C(1)$. $Vol_n$ refers to the $n$-dimensional volume of subspaces in $M_\real$.
\end{mtheorem}

\vspace{0.5cm}
\begin{remark}
 $q$ is defined in Definition \ref{defxi}. The constants $c_{q,n}$ and $c_{\rho,n}$ are given in \S3 Theorem \ref{theorem}. 
\end{remark}

\begin{remark}
The sum $\sum_{\rho \in C(1)}c_{\rho,n}Vol_{n}(H_\rho)$ in the expansion, with $Vol_n(H_\rho)$ being the volume of various facets, comes from the orbifold structures of related toric spaces.\\
\end{remark}
\begin{remark}
The above result can be considered purely as a problem concerning the cone $C^\lor$: When the direction is close enough to the canonical one minimizing the volume, we can have an estimate of the first nontrivial term in terms of the volume. 
\end{remark}
\begin{remark}
 In \cite{KL07}, Chan and the first author have studied a family of Yau's inequalities on Fano toric manifolds and their implications in the lattice points counting problem.
\end{remark}

We give the proof of the theorem in \S 2, omitting the computations that arise from the presence of  orbifold singularities. The orbifold computations will be handled in \S 3.
\section{Proof of theorem}
Before giving the proof of our main theorem, we give a proof of the statement $a_{n-2} \geq \frac{(3n+2)(n-1)n}{24(n+1)} Vol(P)$ for the \textit{balanced} reflexive Delz\'ant polytope $P$ mentioned in the introduction. Recall (see e.g. \cite{TGB}) that $P$ determines a Fano toric manifold $X=X_P$ and $p_P(k)=dim_\comp \Gamma(X,(K_X^{-1})^{\otimes k})=\chi(X,(K_X^{-1})^{\otimes k})$. Using the Riemann-Roch formula, we have 
$$
\begin{array}{rl}
p_P(k)=& \int_X ch((K_X^{-1})^{\otimes k}) Td(X)\\
=&(\int_{X} c_1^n)\frac{k^n}{n!} + (\half \int_{X}c_1^n)\frac{k^{n-1}}{(n-1)!}+\frac{1}{12}\lbrack \int_{X} c_1^n + \int_{X} c_1^{n-2}c_2 \rbrack \frac{k^{n-2}}{(n-2)!}\\
&+\mathcal{O}(k^{n-3}),
\end{array}
$$ where $c_i=c_i(X)$ is the $i$th Chern class of $X$. Since $P$ is balanced, $X$ has a K\"ahler-Einstein metric by the result of Wang-Zhu in \cite{WZ04}. Then we can use the Yau's inequality, $$\int_X c_2c_1^{n-2} \geq \frac{n}{2n+2} \int_X c_1^n$$(see appendix \ref{pointchernnoieq}), and $$\int_X c_1^n=n! Vol(P)$$  to get estimate $$a_{n-2} \geq \frac{(3n+2)(n-1)n}{24(n+1)} Vol(P).$$\\

\begin{remark}
The Yau's inequality and its consequences for algebraic geometry was studied by S.-T. Yau in \cite{Yau77}, as a consequence of the existence of K\"ahler-Einstein metrics. The existence of such metrics was proved for the negative first Chern class case independently by T. Aubin in \cite{Aub} and S.-T. Yau in \cite{Yau78}. The  zero first Chern class case was proven by S.-T. Yau in \cite{Yau78}.
\end{remark}

The proof given below is in a similar flavour. We have to construct some spaces and a line bundle that count the number of lattice points. Extra difficulties arise from the orbifold structure.\\

As mentioned in the introduction, we consider the counting problem for a cone $C^\lor$. Let $C^{\lor} \subset M_\real$ be a cone; we choose an affine hyperplane $H_\xi=\{x \in M_\real|(x,\xi) = 1\}$ by choosing a $\xi \in C \subset N_\real$ such that it cuts the cone cleanly. Move the hyperplane toward infinity by changing $\xi$ to $\xi/k$ and count the number of lattice points bounded below the hyperplane. We want to study the effect of turning the hyperplane to a different angle. \\

In order to make this comparison, we need to have a good parameter space of the hyperplanes with respect to the cone. This is possible if we have a $\mathbb{Q}$-Gorenstein cone.\\

Let $C^\lor \subset M_\real$ be a top dimensional rational cone, $C$ be its dual, $int(C)$ be the interior of $C$, and $C(1)$ be the set of rays in $C$ inward normal to facets in $C^\lor$. For each $\rho \in C(1)$, we let $v_\rho$ be the primitive vector in $\rho$.\\

\begin{defin}
$C^\lor$ is said to be a $\mathbb{Q}$-Gorenstein cone if it satisfies the following two conditions.
\begin{description}
\item [i] (smoothness) For each face $F \subset C^\lor$, the subset of $C(1)$ normal to $F$ can be extended to the $\inte$-basis of $N$.
\item[ii] ($\mathbb{Q}$-Gorenstein) There exists $\lambda \in M$ and some $l \in \inte_{>0}$ such that $(\lambda,v_\rho)=-l$ holds for all  $\rho \in C(1)$.
\end{description}
\label{defCYcone}
\end{defin}

\begin{defin}
Fixing a primitive vector $\xi \in int(C)\cap M$, let $\Delta_{\xi}=\{x \in C^\lor | (x,\xi) \leq 1\}$ and define the lattice points counting function as $$n_\xi(k) = \# (k\Delta_\xi \cap M).$$ For every chosen $\xi$, a ratio $q$ is defined by the equality $(\lambda,\xi)=-ql$.
\label{defxi}
\end{defin}

\vspace{0.5cm}
The above notations and definitions are summarized in the following figure.
\begin{figure}[h]
\includegraphics[scale=.85]{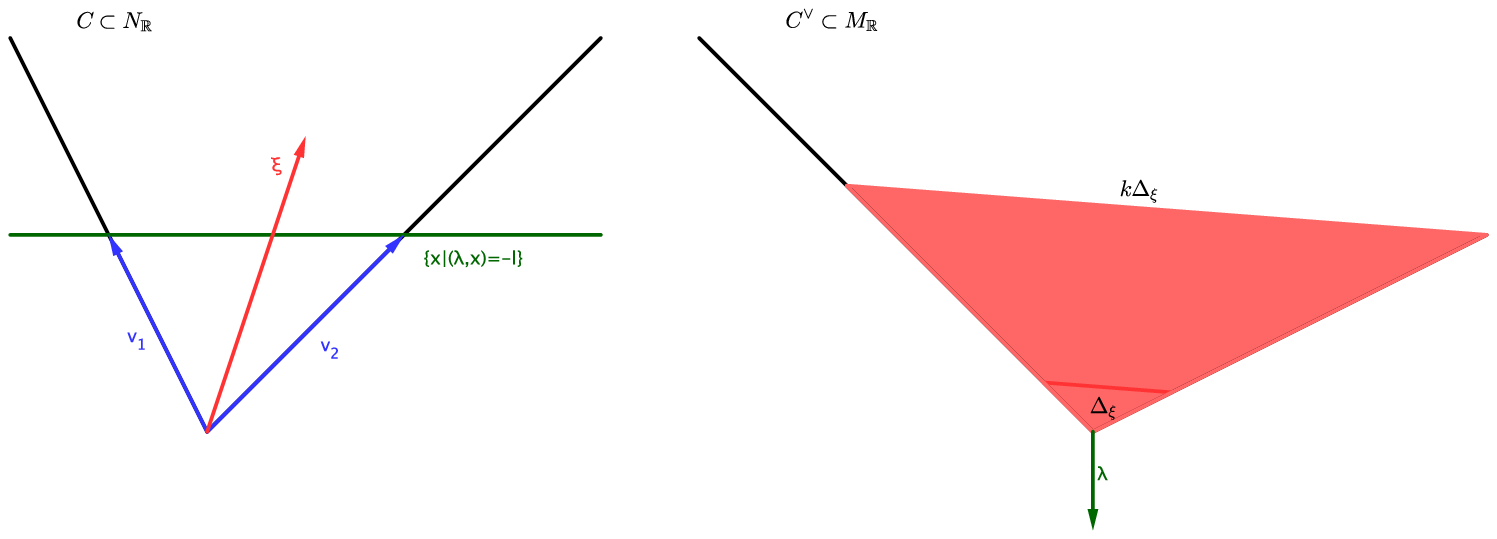}
\caption{}
\label{fig:conefigure}
\end{figure}

$n_\xi(k)$ is the counting function we are interested in; we associate a non-compact toric manifold $Y_C$ with a $\comp^*$-action to each chosen $\xi$ and relate the counting function to some geometric invariants of $Y_C$. We can define $\mathcal{W}$ as $Y_C/\comp^*$, and compactify $Y_C \rightarrow \mathcal{W}$ as a $\mathbb{P}^1$ bundle $\pi : \mathcal{X} \rightarrow \mathcal{W}$.  It turns out that $$n_\xi(k) = \chi(\mathcal{X},\mathcal{L}^{\otimes k })$$\\
for some toric line bundle $\mathcal{L}$ on $\mathcal{X}$.\\

\begin{remark}
$Y_C$ is related to Sasakian geometry. A quick review is given in the appendix.\\
\end{remark}

For example, if $C = cone(e_1,\dots,e_{n+1}) \subset \real^{n+1}$ is the standard cone and $\xi$ is chosen to be $e_1+\dots+e_{n+1}$, then we have $Y_C= \comp^{n+1}-\{0\}$, $\mathcal{W} = \mathbb{P}^n$, and $\mathcal{X}=\mathbb{P}_{\mathcal{W}}(\mathcal{O}(-1) \oplus \mathcal{O})$. $\mathcal{L}$ is the relative $\mathcal{O}(1)$ bundle for the map $\mathcal{X} \rightarrow \mathcal{W}.$\\

When $\mathcal{X}$, $\mathcal{W}$ are smooth, we have $$\chi(\mathcal{X},\mathcal{L}^{\otimes k })=\int_\mathcal{X} ch(\mathcal{L}^{\otimes k}) Td(\mathcal{X})$$ and we get an expression of $b_i$'s in terms of integrals of Chern classes on $\mathcal{W}$. In particular, $b_{n-1}$ is expressed as a combination of $Vol(\Delta_\xi)$ and $$\int_\mathcal{W} c_2(\mathcal{W}) c_1(\mathcal{W})^{n-2}.$$ 
\\
Furthermore, if we are lucky enough that $\xi$ is parallel to $\xi_c$, the above term will have a lower estimate in terms of $Vol(\Delta_\xi)$. We indeed have a K\"ahler structure on $Y_C$ that is transversal K\"ahler-Einstein (appendix \ref{tranKE}). In that case, this K\"ahler structure will induce a K\"ahler-Einstein structure on $\mathcal{W}$. So we can use the Yau's inequality (appendix \ref{pointchernnoieq}) to estimate it in terms of $Vol(\Delta_\xi)$.\\

In general, $\xi_c$ may not even be rational. However, transversal Chern classes are defined and the inequality still holds. In that case, we know the transversal Yau's inequality is strict from Lemma \ref{uniformization}. For a primitive vector $\xi$ such that its direction is close enough to $\xi_c$, we still have our lower estimate by continuity. 

\begin{remark}
$\mathcal{X}$ and $\mathcal{W}$ are orbifolds in most cases.
\end{remark}

\vspace{0.5cm}

Let us begin by giving some notations and definitions concerning the spaces mentioned above. We define a K\"ahler manifold $Y_C$ as follows: There is a map
$$
\inte^{C(1)} \rightarrow  N$$
given by the assignment
$$e_\rho \mapsto  v_{\rho}.$$

\noindent Tensoring with $\comp$ and taking the quotient gives rise to a group homomorphism
$$\ctorus^{C(1)} \rightarrow T_N^\comp,$$ and let $\kappa_C$ be the kernel. Then $Y_C$ is given by the G.I.T. quotient of $\comp^{C(1)}-\{0\}$ by $\kappa_C$ via the natural action of $\ctorus^{C(1)}$ (see e.g. \cite{AB01}).\\

A similar construction using the symplectic quotient gives a symplectic structure on $Y_C$. In general, any $\xi \in N_\real$ gives a vector field $\xi^{\#}$ on $Y_C$ by the real torus action. We also denote $\xi^{1,0}=\sqrt{-1}\xi^{\#}-J\xi^{\#}$ to be the corresponding holomorphic vector field. If $\xi$ is primitive, we have a $\comp^*$-action on $Y_C$ given by the holomorphic vector field. \\

We define $\mathcal{W}=Y_C/\comp^*$. Hence $Y_C$ can be viewed as a $\comp^*$-bundle over $\mathcal{W}$. Using the standard action of $\comp^*$ on $\comp$, we can associate a line bundle $L_\mathcal{W}$ over $\mathcal{W}$. We let $\mathcal{X}=\mathbb{P}_\mathcal{W}(L_\mathcal{W}\oplus \mathcal{O})$ be the $\mathbb{P}^1$-bundle over $\mathcal{W}$. There is a relative $\mathcal{O}(1)$ bundle on $\mathcal{X}$, denoted by $\mathcal{L}$, associating to the natural projection map $\pi:\mathcal{X} \rightarrow \mathcal{W}$. \\

From the symplectic perspective, the space $Y_C$ is related to another compact odd dimensional space. If we let $r:Y_C \rightarrow \real$ be a smooth function such that $\half r^2$ is the moment map of the induced $\mathbb{S}^1$-action, then $S:=\{r=1\} \subset Y_C$ is a principal $\mathbb{S}^1$-bundle over $\mathcal{W}$. $\mathcal{W}$ can be viewed as the symplectic quotient of $Y_C$ via the $\mathbb{S}^1$-action. This gives $\mathcal{W}$ a structure of K\"ahler orbifold. \\

$(S, \xi^{\#}|_S)$ is indeed a Sasakian manifold and $(Y_C, \xi^{\#},\omega,J)$ can be viewed as its (K\"ahler) cone manifold. The relationship between $(S, \xi^{\#}|_S)$ and $(Y_C, \xi^{\#},\omega,J)$ is a one to one correspondence. Furthermore, $(S, \xi^{\#}|_S)$ is still defined even when $\xi$ is an irrational vector. The transversal Yau's inequality still holds when $(S, \xi^{\#}|_S)$ admits a transversal K\"ahler-Einstein metric. For details, we refer readers to \cite{AB09} and \cite{FOW09}. \\

From the toric perspective, in case $\xi$ is primitive, we can complete the cone $C \subset N_\real$ to a complete fan $\Sigma$ by adding $\xi$ and $-\xi$ to it. We can also define a quotient lattice $N'=N/\inte \xi$ and $p: N_\real \rightarrow N_\real'$. The image of $C$ together with its faces form a fan $\Sigma' \subset N_\real'$. Then we have two complete fans $\Sigma$ and $\Sigma'$, having $|C(1)|+2$ and $|C(1)|$ rays respectively. The spaces $\mathcal{X}$ and $\mathcal{W}$ are the toric orbifolds associated to the fans $\Sigma$ and $\Sigma'$ respectively, using the G.I.T. quotient construction (see \cite{AB01}). The line bundle $\mathcal{L}$ is identified with the divisor bundle of $D_{-\xi}$.\\

Using the toric perspective, we have (see e.g. \cite{TGB}) $$n_\xi(k) = dimH^0(\mathcal{X}, \mathcal{L}^{\otimes k}).$$ In order to get $$n_\xi(k)=\chi(\mathcal{X},\mathcal{L}^{\otimes k}),$$ we use the convexity of the supporting function corresponding to the line bundle $\mathcal{L}$ and the Demazure vanishing theorem \cite{TGB}.\\
Using the Kawasaki-Riemann-Roch formula in \cite{KW79}, we get the expansion
$$\begin{array}{rl}
n_\xi(k)=&\int_{\mathcal{X}} ch(\mathcal{L}^{\otimes k}) Td(\mathcal{X})+ R_{orb}\\
\\
=&b_{n+1} k^{n+1} + b_n k^n + b_{n-1} k^{n-1}+ \cdots.\\
\end{array}$$

\noindent Here $R_{orb}$ is the contribution from orbifold singular strata (starting from codimension 2). Presence of $R_{orb}$ results in periodicity of $b_i$'s in $k$ (for $i\leq n-1$); more explicitly, $b_i(k)$'s are composition of rational functions with functions of the form $e^{\frac{2\pi \sqrt{-1} ck}{N}}$. The contribution of $\int_{\mathcal{X}} ch(\mathcal{L}^{\otimes k}) Td(\mathcal{X})$ to the coefficient $b_{n-1}$ only involves integrals of products of $c_1(\mathcal{W})$, $c_1(L_\mathcal{W})$, and $c_2(\mathcal{W})$. \\

First, there is an equality relating $c_1(\mathcal{W})$ and $c_1(L_\mathcal{W})$. Let $\xi^{1,0}$ be the holomorphic vector field on $Y_C$ associated to $\xi$, and let $L=\comp \xi^{1,0}$ be the trivial line bundle over $Y_C$ with a Hermitian metric $\half r^2$. Its quotient by $\comp^*$ gives a Hermitian metric on $L_\mathcal{W} \rightarrow \mathcal{W}$ with $c_1(L_\mathcal{W}^*) = \frac{1}{2\pi} [d\eta]$ on $\mathcal{W}$, where $\eta=d^c log(r)$. Here $\eta$ is the contact $1$-form of the Sasakian manifold $S$, and $d\eta$ descends to $\mathcal{W}$ (see e.g. \cite{FOW09}). $c_1(\mathcal{W})$ is also related to the class $[d\eta]$, and there is a equality $$c_1(\mathcal{W}) =\frac{q}{2\pi}[d\eta]=qc_1(L_\mathcal{W}^*),$$ where $q$ is the ratio defined in Definition \ref{defxi}.\\

The integral involving only the first Chern class is given in \cite{MSY06} by
$$\int_{\mathcal{W}} c_1(L_\mathcal{W}^*)^{n}
=(\frac{1}{2\pi})^{n+1} \int_S (d\eta)^n \wedge \eta
=(n+1)! Vol (\Delta_\xi).
$$\\
Second, there is a term 
\begin{equation}\int_\mathcal{W} c_2(\mathcal{W}) c_1(\mathcal{W})^{n-2}=\int_S c_2^{B}(S)c_1^B(S)^{n-2} \eta
\label{transChernclass}
\end{equation}\\
in the expression of $b_{n-1}$. Here $c_k^B(S)$'s are basic Chern classes of $S$ defined in \cite{FOW09}.\\

A key observation is that this term can be controlled if $\xi$ is suitably chosen: for our cone $C^\lor$, there is a canonical direction $\xi_c$ associated to it. (It is the unique minimizer of the volume function $F(\xi) = Vol(\{y \in C^\lor| (\xi,y) \leq 1\})$ restricting to $I=\{x \in int(C)|(x,\lambda)=-(n+1)l\}$.) For a general $\xi$ (may not be rational), $Y_C$ and $S$ are still defined and we can discuss the transversal K\"ahler geometry of $S$, even though the quotient $\mathcal{W}$ may not exist. Those $\xi$'s parallel to $\xi_c$ are exactly those with $S$ having a transversal K\"ahler-Einstein metric. In that case, we can obtain a lower bound of \eqref{transChernclass} by the following transversal Yau's inequality.\\
 \begin{lemma}[Transversal Yau's inequality]
Let $(S,g,J)$ be a Sasakian manifold of dimension $2n+1$ such that its transversal Ricci form satisfies $$Ric^T = \tau (\half d\eta)$$ for some $\tau \in \real$. Then
\begin{equation}
\int_S \lbrack c_2^B(S)-\frac{n}{2(n+1)}(c_1^B(S))^2\rbrack \wedge (\half d\eta)^{n-2}\wedge\eta\geq 0.
\label{tranIneq}
\end{equation}
If the equality sign holds, then the Einstein metric has constant transversal holomorphic bisectional curvature.
\end{lemma}

As basic Chern classes depend only on the Reeb vector field $\xi^{\#}$ (or equivalently, the transversal complex structure) and not on the metric, the inequality holds whenever the Reeb vector field is given by $\xi_c$. It can be argued that it also holds for those $\xi$'s parallel to $\xi_c$. For details, readers may consult the appendix.\\

In the case that $\xi_c$ is not a rational vector, the following uniformization lemma (Lemma \ref{uniformization}) tells us that we must have a strict inequality for those $\xi$'s parallel to $\xi_c$. Hence, for primitive $\xi$ having its direction close enough to $\xi_c$, we still have the same estimate.\\

The remainder of this section is devoted to prove Lemma \ref{uniformization}. It is a statement about toric Sasakian geometry. Readers may skip the proof and progress to the next section, where we will deal with the orbifold's contribution $R_{orb}$.\\

As mentioned in the appendix, given a $\mathbb{Q}$-Gorenstein cone $C^\lor$, we can associate to it a unique toric Sasaki-Einstein manifold $(S,g,J_c)$ with the Reeb vector field given by $\xi_c$. For this space, the Yau's inequality holds and there is a uniformization result:

\begin{lemma}
If equality holds in \eqref{tranIneq} for the space $(S,g,J_c)$, then $|C^\lor (1)|=n+1$ and $\xi_c$ is a rational vector.
\label{uniformization}
\end{lemma}
\begin{proof}
\noindent
\begin{description}
\item [Step 1]
Since $C(S)=:Y$ is Ricci flat, we have a pointwise Yau's inequality for $Y$: 
$$c_2(Y)\wedge \omega^n \geq 0.$$

From the fact that $S$ has constant transversal holomorphic bisectional curvature, we have that $$ \lbrack c_2^B-\frac{n}{2(n+1)}(c_1^B)^2\rbrack \wedge (\half d\eta)^{n-2}\wedge\eta=0.$$ A second fundamental form computation implies $c_2(Y)\wedge \omega^n=0$. This further says $Y$ is flat and $S$ has positive constant sectional curvature.
\\
\item [Step 2]
Let $\tilde{S}$ be the universal cover of $S$ and $\tilde{Y}=C(\tilde{S})$, which is a finite covering of degree $N=| \pi_1(Y)|=|\pi_1(S)|$ as $S$ is Ricci positive. We lift the Sasakian structure to $\tilde{S}$ and hence the K\"ahler structure to $\tilde{Y}$. We let $p :\tilde{Y} \rightarrow Y$ be the covering map.

The torus action can be lifted to $\tilde{Y}$ (may be non-effective), and we have
$$\begin{array}{ccc}
\torus^{n+1} &\curvearrowright & \tilde{Y}\\
\phi \downarrow& & \downarrow\\
\torus^{n+1} & \curvearrowright & Y,
\end{array}$$
where $\phi$ is multiplication by $N$. We define $\tilde{\xi}_c \in \tilde{\mathfrak{t}}$ by letting $\phi_*(\tilde{\xi}_c)=\xi_c$.\\

\item [Step 3] We may assume $Y$ is simply connected with constant holomorphic bisectional curvature by considering $\tilde{Y}$ instead. $S$, as a Riemannian manifold, is identified with the $2n+1$ dimensional sphere $\mathbb{S}^{2n+1}$. To identify the Sasakian structure, it suffices to identify the Killing vector field $K_c$ with $K_{std}$, generated by $\xi_c$ and $\xi_{std}$ respectively. Fixing a point $p_0 \in S$, we can choose isometry between $S$ and $\mathbb{S}^{2n+1}$, which identify $(T_{p_0}S, K_c(p_0), \nabla K_c (p_0))$ with $(T_1\mathbb{S}^{2n+1}, K_{std}(1),\nabla K_{std} (1))$, for some point $1\in \mathbb{S}^{2n+1}$. This identifies $K_c$ with $K_{std}$.\\

\item [Step 4] We have a possibly non-standard action $\torus^{n+1}\curvearrowright \mathbb{S}^{2n+1}$. The flow line of $K_c$ closes up, this shows the rationality of $\xi_c$.  Taking the quotient, we have $\torus^{n} \curvearrowright (\comp \mathbb{P}^n,\omega_{std})$ being a toric K\"ahler manifold. Hence conjugation by automorphism of $(\comp \mathbb{P}^n,\omega_{std})$ gives the standard action. In particular, the moment map image is a cone with $n+1$ rays.
\end{description}
\end{proof}

\section{Riemann-Roch for orbifolds}
This section is devoted to the Riemann-Roch computation for orbifolds.\\

For each face $\tau \subset C$, we define $$\tau^1=cone(\tau \cup \{\xi\}), \tau^{-1}=cone(\tau \cup \{-\xi\}).$$ Here $cone(F):=\{ \sum_i a_i v_i | v_i \in F, a_i \geq 0\}$ is the cone generated by the vectors in $F$, where $F$ is a subset of a vector space. We let $\Sigma$ be the fan consisting of all $\tau$, $\tau^1$, $\tau^{-1}$ for all faces $\tau \subset C$. $\Sigma$ is the normal fan for the polytope $\{y \in C^\lor | r_1 \leq (\xi,y) \leq r_2 \}$ for any $r_2 >r_1>0$.\\

As in the previous section, we have an orbifold $\mathcal{X}$, together with an orbi-line bundle $\mathcal{L}$, which compute the function $n_\xi(k)$:
\begin{equation}
n_\xi(k)= \chi(\mathcal{X},\mathcal{L}^{\otimes k})=\int_{\mathcal{X}} ch(\mathcal{L}^{\otimes k}) Td(\mathcal{X})+ R_{orb}.
\end{equation}
The formula for $R_{orb}$ is given in \cite{GU97}. We will concentrate on the contribution of $R_{orb}$ to the coefficient $b_{n-1}$. For that purpose, we consider the codimension two singular stratum of $\mathcal{X}$ (since there is no singular stratum of codimension one).\\

The space $\mathcal{X}$ is a quotient of affine space $\comp^{\Sigma(1)}$ ($\Sigma(1)$ stands for the set of all rays in the fan $\Sigma$) by some subgroup $\kappa_\Sigma$ of $\ctorus^{\Sigma(1)}$ via the quotient construction mentioned in \cite{AB01}.  Each ray in the fan corresponds to a coordinate of the affine space $\comp^{\Sigma(1)}$ before taking the quotient. Hence vanishing of some of the coordinate functions defines a closed sub-orbifold (may be non-effective) of the quotient. For example, codimension one sub-orbifolds that correspond to rays in $\Sigma(1)$ are the toric divisors. \\

Codimension two closed toric sub-orbifolds are defined by two rays. For each $\rho \in C(1)$ and $\alpha$ ($\alpha=\pm 1$), we have a closed sub-orbifold $F^\alpha_{\rho}$ given by vanishing of the coordinate functions corresponding to the rays $\rho$ and $\alpha \xi$. These give all the singular strata necessary for the computation of the coefficient $b_{n-1}$.\vspace{0.3cm}\\

To obtain a coordinate chart, we can choose an orbifold chart by taking any maximal cone  $\tau \supset \rho$. By the smoothness assumption of the cone $C^\lor$, we can take a $\mu \in N$ together with the primitive vectors of the rays $\{v_{\rho'} | \rho' \in \tau(1) \}$ to be a $\mathbb{Z}$-basis of $N$ and write $\xi$ as $\xi = \sum c_{\rho'} v_{\rho'} + d \mu$, for some $c_{\rho'} ,d \in \mathbb{Z}$. The cone $\tau^\alpha$ gives an orbifold chart $\mathbb{Z}_d \curvearrowright \comp^{\tau(1)}\times \comp$ of $\mathcal{X}$. $\comp^{\tau(1)}\times \comp \hookrightarrow \comp^{\Sigma(1)}$ (the last coordinate corresponds to $\alpha \xi$) is a subset given by letting the coordinate functions corresponding to the rays other than $\{\alpha \xi\} \cup \tau(1)$ be $1$. The group $\mathbb{Z}_d$ is the subgroup of $\kappa_\Sigma$ which preserve the subset $\comp^{\tau(1)}\times \comp$. Then $\comp^{\tau(1)}\times \comp$ cover a dense open subset of $\mathcal{X}$. A local chart of $F^\alpha_{\rho}$ is given by vanishing of coordinate functions that correspond to $\rho$ and $\alpha \xi$. \vspace{0.3cm}\\

If we let $\hat{d} = d/ \lbrack g.c.d. (\{c_{\rho'}\}_{\rho' \neq \rho },d) \rbrack$ and $\Gamma_{\rho}^\alpha=\mathbb{Z}_{\hat{d}}\leq\mathbb{Z}_d$, then $\Gamma_{\rho}^\alpha$ acts trivially on $F^\alpha_{\rho}$. Let $\theta$ be the induced action of $\Gamma_{\rho}^\alpha$ on $\mathcal{L}|_{F_{\rho}^\alpha}$ ($\eta \in \Gamma_{\rho}^\alpha$ acts by multiplication by $\theta(\eta)$). These are the combinatorial data we needed for our computations.\vspace{1cm}\\

According to the Kawasaki-Riemann-Roch formula in \cite{GU97}, we have 
$R_{orb}=\sum_{\rho \in C(1)} \sum_\alpha KRR(\rho,\alpha,\mathcal{L}^{k})+\mathcal{O}(k^{n-2}),$
where
\[\begin{array}{ll}
&KRR(\rho, \alpha,\mathcal{L}^{\otimes k})\\
=&\displaystyle \sum_{\eta \in \Gamma_{\rho}^\alpha-\{0\}}\frac{\theta(\eta)^k}{(1-\eta^{-c_{\rho}})(1-\eta^\alpha)}\int_{F_{\rho}^\alpha} \frac{c_1(\mathcal{L})^{n-1}}{(n-1)!} k^{n-1}
 + \mathcal{O}(k^{n-2}).\\
\end{array}\]
Recall that $q \in \mathbb{Q}$ is the ratio such that $(\xi, \lambda) =-ql$ ($\lambda$ as in definition \eqref{defCYcone}). For $n \geq 2$, we have
\begin{equation*}
	\begin{array}{rl}
	b_{n+1}=&\displaystyle \frac{1}{(n+1)!} \int_{\mathcal{W}} c_1(L_\mathcal{W}^*)^n\\
	\\
		 =&\displaystyle Vol_{M_\real}(\Delta_\xi)\\
		\\

	b_n=&\displaystyle \frac{1}{2n!} \int_{\mathcal{W}}\lbrace c_1(L_\mathcal{W}^*)^n+ c_1(L_\mathcal{W}^*)^{n-1} c_1(\mathcal{W})\rbrace\\
	\\
	=&\displaystyle \frac{1+q}{2}(n+1) Vol_{M_\real}(\Delta_\xi)\\
	\\
	b_{n-1}=&\displaystyle \frac{(q^2+3q+1)}{12(n-1)!}\int_{\mathcal{W}} c_1(L_\mathcal{W}^*)^n 
	+\frac{1}{12(n-1)!}\int_{\mathcal{W}} c_2(\mathcal{W})c_1(L_\mathcal{W}^*)^{n-2}  \\
	\\
	&\displaystyle +\frac{1}{(n-1)!}\sum_{\rho,\alpha} \sum_{\eta \in \Gamma_\rho^\alpha-\{0\}}\frac{\chi(\eta)^k}{(1-\eta^{-c_\rho})(1-\eta^\alpha)}\int_{F_\rho} c_1(L_\mathcal{W}^*)^{n-1}\\
	\\
	=&\displaystyle \frac{(q^2+3q+1)}{12n(n+1)} Vol_{M_\real}(\Delta_\xi)+\frac{1}{12(n-1)!}\int_{\mathcal{W}} c_2(\mathcal{W})c_1(L_\mathcal{W}^*)^{n-2}  \\
	\\
	&\displaystyle + \sum_{\rho \in C(1)} \frac{n}{|v_\rho|} \sum_{\eta \in  \Gamma_\rho^\alpha-\{0\}} \frac{1-\eta^{k
	+1}}{(1-\eta^{-c_\rho})(1-\eta)}  Vol_{n}(H_\rho)

	\end{array}
	\end{equation*}

\vspace{0.5cm}
\begin{remark}
\begin{enumerate}

\item The equality 
		$$\int_{F_\rho} c_1(L_\mathcal{W}^*)^{n-1}=\frac{n!}{|v_\rho|} Vol_{n}(H_\rho)$$
is similar to that in the previous section.
\vspace{0.5cm}

\item For the case $n=1$, the formula reads
$$
\begin{array}{ll}
	\chi(\mathcal{X},\mathcal{L}^{\otimes k}) 
	
	=&\displaystyle  ( k^2 + (1+q) k + q) Vol(\Delta_\xi) \\
	\\
	&\displaystyle + \sum_{\rho=\{\rho_1,\rho_2\}} \frac{1}{|v_\rho|} [\sum_{\eta \in  \Gamma_\rho^\alpha-\{0\}} \frac{1-\eta^{k
	+1}}{(1-\eta^{-c_\rho})(1-\eta)} ] Vol_1(H_\rho)
\end{array}
$$
\end{enumerate}
\end{remark}
We give a $2$-dimensional example to illustrate the contribution from orbifold singularities.

\begin{ex*}
Letting $N_\real =\real^2$, $C=cone(e_1,-e_1+3e_2)$, and $\xi = (1,1) \in \real^2$,
we have $\lambda = (-3,-2) \in (\real^2)^*$ and $q= \frac{5}{3}$. This gives $\Delta_\xi = cone\{(0,0),(0,1),(\frac{3}{4},\frac{1}{4})\}$, $Vol(\Delta_{\xi})=\frac{3}{8}$, $H_{e_1}= cone\{(0,0),(0,1)\}$, and $H_{-e_1+3e_2}= cone\{(0,0),(\frac{3}{4},\frac{1}{4})\}$ .

\vspace{0.5cm}
In this case, the lattice points counting function is 
$$\eta_\xi(k) = \frac{1}{16}\lbrace6k^2 +16k + 2(\sqrt{-1})^k\lbrack1+(-1)^k\rbrack + (-1)^k +11\rbrace.$$
\end{ex*}
\vspace{0.5cm}

Combining with the above sections, we have our main theorem:

\begin{mtheorem}
Given an $(n+1)$-dimensional $\mathbb{Q}$-Gorenstein cone $C^\lor \subset M_\real$,
with its canonical Reeb vector $\xi_c \in C \subset N_\real$,
if $\xi_c$ is rational, let $\xi \in N$ be the primitive vector parallel to $\xi_c$; otherwise, choose $\xi$ having its direction close enough to $\xi_c$. If we write\\
$$n_\xi(k)=b_{n+1}k^{n+1}+b_n k^n +b_{n-1}k^{n-1}+\mathcal{O}(k^{n-2}),$$\\
then we have\\
$$\begin{array}{rl}
 b_{n+1} = &\displaystyle Vol_{n+1}(\Delta_\xi)\\
 \\
 b_n =&\displaystyle \frac{1+q}{2}(n+1) Vol_{n+1}(\Delta_\xi)\\
 \\
 b_{n-1} \geq &\displaystyle \frac{n}{24} \lbrack q^2(3n+2)+2(3q+1)(n+1)\rbrack Vol_{n+1}(\Delta_\xi)\\
 \\
 &\displaystyle + \sum_{\rho \in C(1)} \frac{n}{|v_\rho|} [\sum_{\eta \in  \Gamma_\rho^\alpha-\{0\}} \frac{1-\eta^{k
	+1}}{(1-\eta^{-c_\rho})(1-\eta)} ] Vol_{n}(H_\rho)

\end{array}$$
\label{theorem}
\end{mtheorem}

\section{Appendix: Sasakian Geometry}
\subsection{Basic results and notations from Sasakian geometry}

We recall some definitions in Sasakian geometry, following \cite{FOW09} and \cite{MSY06}.\\

\begin{defin}
A Sasakian manifold is a Riemannian manifold $(S,g)$ of real dimension $2n+1$, together with a choice of $\real$-invariant complex structure $J$ on its cone manifold $(C(S),g^{C(S)}):=(S\times\realp,dr^2+r^2g)$, such that it is K\"ahler.
\end{defin}

Given a Sasakian manifold, we let $V=r\dd{r}$ be the Euler vector field and $K=JV$ be the Reeb vector field. 
There is an equivalent definition of Sasakian manifold from the symplectic aspect given in \cite{AB09}, which works better with toric geometry. We will freely interchange between the two definitions.\\




We follow the notations of transversal K\"ahler geometry of a Sasakian manifold introduced in \cite{FOW09}.\\

\begin{defin}
A Sasakian manifold $(S,g,J)$ is said to be transversal K\"ahler-Einstein if there is a real constant $\tau$ such that
\begin{equation}
Ric^{T}=\tau(\frac{1}{2}d\eta)
\end{equation}
where $\eta$ is the contact $1$-form on $S$.
\label{tranKE}
\end{defin}

\begin{remark}
A Sasakian manifold is transversal K\"ahler-Einstein (with constant $\tau=2n+2$) if and only if its cone manifold $C(S)$ is Ricci flat. 
\end{remark}
\vspace{0.5cm}

To obtain a Yau's inequality for transversal K\"ahler-Einstein manifolds, we recall

\begin{theorem}[Yau's inequality for K\"ahler-Einstein manifolds]
Let $(M,\omega,J,g)$ be a connected K\"ahler-Einstein manifold of complex dimension $n$. Then
\begin{equation}
\lbrack c_2(\nabla^{l.c.})-\frac{n}{2(n+1)}c_1(\nabla^{l.c.})^2\rbrack \wedge \omega^{n-2}= \delta\omega^{n}
\end{equation}
for some positive function $\delta$. $\delta \equiv 0$ if and only if $M$ has constant holomorphic bisectional curvature, where $\nabla^{l.c.}$ is the Levi-Civita connection on $M$ and $c_i(\nabla^{l.c.})$ is the corresponding $i$-th Chern form.
\label{pointchernnoieq}
\end{theorem}

We can have the following lemma, generalizing the above to the transversal K\"ahler-Einstein case.

\begin{lemma}[Transversal Yau's inequality for Sasaki-Einstein manifolds]
Let $(S,g,J)$ be a Sasakian manifold of real dimension $2n+1$, which satisfies $$Ric^T = \tau (\half d\eta)$$ for some $\tau$. Then
\begin{equation}
\int_S \lbrack c_2^B(S)-\frac{n}{2(n+1)}(c_1^B(S))^2\rbrack \wedge (\half d\eta)^{n-2}\wedge\eta\geq 0,
\label{tranIneq}
\end{equation}
where $c_i^B(S)$'s are the basic Chern classes of $S$ defined in \cite{FOW09}.
If equality holds, then the Einstein metric has constant transversal holomorphic bisectional curvature.
\end{lemma}

\begin{remark}
The above integral is independent of basic deformations of Sasakian structures described in \cite{FOW09}.
\end{remark}

\vspace{0.5cm}
The existence of such a metric in the case $\tau>0$ is what we are interested in. First, we normalized the constant $\tau$ by $D$-homothetic transformation to get a new K\"ahler metric $g'$ with Einstein constant $2n+2$.\\

Given a Sasakian manifold $(S,g,J)$ and $\alpha \in \real_{>0}$, define $K'= \frac{1}{\alpha} K$ and $g'=\alpha g + \alpha(\alpha-1) \eta \otimes \eta$. The complex structure on $C(S)$, $J'$, is given by 
	$$\begin{array}{lll}
	J' (V) &=& K'\\
	J'(Y) &=& J(Y)\;\;\;for\;\;Y \in \Gamma(T_S)^{K^{\perp}}\\
	\end{array}$$

\vspace{0.5cm}
Then we have the new K\"ahler form $\omega' = \alpha \omega$ and the transversal K\"ahler form $\half d \eta ' = \alpha (\half d\eta)$.\\

Notice that $g'$ has constant transversal holomorphic bisectional curvature if and only if $g$ does (with different constants).\\

\vspace{0.5cm}
The existence of transversal K\"ahler-Einstein structures is proven in the toric case by Futaki-Ono-Wang in \cite{FOW09}.

\subsection{Toric Sasakian geometry and existence of transversal Einstein metrics}
For notations of toric Sasakian geometry, we refer readers to \cite{AB09}.

\begin{theorem}[Futaki-Ono-Wang \cite{FOW09}]
For a toric Sasakian manifold with $C^\lor$ being its moment cone, if $C^\lor$ is a $\mathbb{Q}$-Gorenstein cone, then there exists a unique $J_c$, which is a torus invariant complex structure, such that the corresponding cone manifold is Ricci flat (or equivalently, the corresponding Sasakian manifold is transversal K\"ahler-Einstein). 

\end{theorem}
\vspace{0.5cm}
The Reeb vector of such a Ricci flat metric is characterized by a volume minimization. Given a $\mathbb{Q}$-Gorenstein cone $C^\lor$, we let $F: int(C) \rightarrow \real$ defined by $F(\xi) = Vol_{M_\real}(\{y\in C^\lor| (\xi,y) \leq 1\})$, and $I= \{x\in C |(x,\lambda)=-(n+1)l\}$. Then there are the following two theorems:

\vspace{0.5cm}
\begin{theorem}[Martelli-Sparks-Yau \cite{MSY06}]
$F|_I$ is strictly convex with a unique minimum point $\xi_c$.
\label{thmxic}
\end{theorem}

\vspace{0.5cm}
\begin{theorem}[Futaki-Ono-Wang \cite{FOW09}]
The Sasaki complex structure $J_c$, which admits transversal K\"ahler-Einstein Sasakian metric, has its Reeb vector $K_c$ equal to the vector field generated by $\xi_c$ via the torus action.
\end{theorem}
\vspace{0.5cm}
For a toric $(2n+1)$-dimensional Sasakian manifold $(S, g,J)$, by taking a $D$-homothetic transformation with constant $\alpha$, we have the new moment map and Reeb vector given by $\mu' = \alpha \mu$ and $K' = \frac{1}{\alpha} K$, respectively.

\vspace{0.5cm}
Hence for each $\xi \in C^o$ parallel to $\xi_c$, there is a unique $J$ having its Reeb vector field generated by $\xi$, which is transversal K\"ahler-Einstein.

\vspace{0.5cm}
As a consequence, for any such $J$, we have 
\begin{equation*}
\int_S \lbrack c_2^B-\frac{n}{2(n+1)}(c_1^B)^2\rbrack \wedge (\half d\eta)^{n-2}\wedge\eta\geq 0
\end{equation*}
and equality holds if and only if $J$ has constant transversal holomorphic bisectional curvature.\\

\section{Acknowledgements}
The authors thank Akito Futaki for useful discussions. The work described in this paper was substantially supported by a grant from the Research Grants Council of the Hong Kong Special Administrative Region, China (Project No. CUHK403709).

\bibliographystyle{amsplain}
\bibliography{Xbib}

\end{document}